%% file: smartgrid_zazo.tex
\documentclass[journal]{IEEEtran}

\usepackage{amsmath,graphicx}
\usepackage{bm}
\usepackage{amsthm,amssymb}
\usepackage{algorithm}
\usepackage{algpseudocode}
\usepackage{subcaption}

\usepackage{enumitem}

\newcommand{\unit}[1]{\ensuremath{\, \mathrm{#1}}}
\newtheorem{theorem}{Theorem}
\newtheorem{lemma}{Lemma}

\DeclareMathOperator{\diag}{diag}

\usepackage{hyperref}
\hypersetup{
  colorlinks   = true,    
  urlcolor     = blue,    
  linkcolor    = blue,    
  citecolor    = red    
}

\usepackage{pgfplots}
\pgfplotsset{compat=newest}
\pgfplotsset{plot coordinates/math parser=false}
\pgfplotsset{major grid style={dashed,line width=0.5pt}}

\newlength\figureheightB
\newlength\figurewidthB
\newlength\figureheightC
\newlength\figurewidthC
\newlength\figureheightD    
\newlength\figurewidthD
\newlength\figureheightE    
\newlength\figurewidthE
\def\reffig{Figure}
\def\x{{\mathbf x}}
\def\a{{\mathbf a}}

\def\betamm{\beta_{m}}
\newcommand{\gameA}{\mathcal{G}_{\bm{\delta}}}
\newcommand{\gameB}{\mathcal{G}_{m}}
\newcommand{\gameBt}{\tilde{\mathcal{G}}_{m}(\bm{\lambda})}

\title{Robust Worst-Case Analysis of Demand-Side Management in Smart Grids}

\author{Javier~Zazo,~\IEEEmembership{Member,~IEEE,}
        Santiago~Zazo,~\IEEEmembership{Member,~IEEE,}
        and~Sergio~Valcarcel Macua,~\IEEEmembership{Student Member,~IEEE}
\thanks{All authors are affiliated with the Universidad Polit\'{e}cnica de Madrid, in Av. Complutense 30, Madrid, 28040, Spain.}
\thanks{Emails: javier.zazo.ruiz@upm.es,\{santiago,sergio\}@gaps.ssr.upm.es}
\thanks{This work was supported in part by the Spanish Ministry of Science and
Innovation under the grant TEC2013-46011-C3-1-R (UnderWorld), the
COMONSENS Network of Excellence TEC2015-69648-REDC and by an FPU
doctoral grant to the first author.
Paper no. TSG-00815-2015.}%
\thanks{Digital Object Identifier \href{http://ieeexplore.ieee.org/xpl/articleDetails.jsp?arnumber=7460956}{10.1109/TSG.2016.2559583}}
}

\IEEEpubid{%
\begin{minipage}[]{1.5\linewidth}%
\centering
\vspace{2em}
1949-3053 \copyright~2016 IEEE. Personal use is permitted, but republication/redistribution requires IEEE permission.\\
See \href{http:/www.ieee.org/publications\_standards/publications/rights/index.html}{http:/www.ieee.org/publications\_standards/publications/rights/index.html} for more information.%
\end{minipage}
}

\begin{document}
\normalsize

\maketitle

\begin{abstract}
Demand-side management presents significant benefits in reducing the energy load in smart grids by balancing consumption demands or including energy generation and/or storage devices in the user's side. 
These techniques coordinate the energy load so that users minimize their monetary expenditure.
However, these methods require accurate predictions in the energy consumption profiles, which make them inflexible to real demand variations. 
In this paper we propose a realistic model that accounts for uncertainty in these variations and calculates a robust price for all users in the smart grid.
We analyze the existence of solutions for this novel scenario, propose convergent distributed algorithms to find them, and perform simulations considering energy expenditure.
We show that this model can effectively reduce the monetary expenses for all users in a real-time market, while at the same time it provides a reliable production cost estimate to the energy supplier.
\end{abstract}


\section{Introduction}
\label{sec:intro}
\IEEEPARstart{S}{mart} grids represent the concept behind the intended evolution of the electric grid, which through information acquisition from the end users will allow for more efficient energy distribution, flexibility in network topology, adaptive load management and better integration of renewable energy plants with the consumption requirements of the users.

One important concept to achieve the smart grid's objectives is the demand-side management, which includes the techniques for better energy efficiency programs, distributed energy generation, energy storage and load management. 
The model presented in this work establishes a robust energy price one day ahead, which takes into account the production costs of energy as well as uncertainties in the expected energy loads.
We also consider a real-time market, where users are charged their specific demanded energy at the robust price, permitting some deviation without any extra penalty.
This model provides some flexibility to every user in their real-time consumption energy loads and calculates the energy prices in a conservative manner, taking into account production costs and possible errors.
In particular, we focus on a robust analysis of the demand-side management algorithms, where the energy strategies account for worst-case error deviations in the predicted loads.
In our formulation, these error terms vary from user to user and depend on the energy load profiles.

We show that this model can effectively reduce the monetary expenses of all users in the real-time market, while at the same time it can provide a confident production cost estimate to the energy supplier.
Our analysis also shows that the added computational cost of these error calculations is minimal and the monetary gain is significant.
Moreover, the calculations are amenable to a distributed implementation within the communications network of the smart grid.
This is important because distributed algorithms are beneficial in terms of communication efficiency and computation scalability~\cite{Bertsekas1997}.

To further support our results we perform simulations in a real-time modeled scenario.
We also compare our robust proposal with a naive version where users do not take into account price deviations, but where the energy supplier still establishes a robust price.
Finally, we also present results illustrating the convergence behavior of our algorithm.

This study can be of special interest to energy supply authorities who can analyze how much energy they should account for if they plan for grid load estimation errors and peak demand costs. 
Even though users may be only interested in their monetary expenditure as energy consumers, supply authorities have to account for local deviations in the distribution network and coordinate with all agents to provide unobstructed service by controlling regular and peak supply stations. 
The fact that an agent in the network can alter the expected energy load by injecting or subtracting energy unexpectedly is a source of extra costs that have to be taken into account by suppliers that monitor the grid's energy demand. 
Therefore, these energy provision costs justify a modified pricing model and a worst-case analysis of the demand-side network.

\IEEEpubidadjcol

\vspace{-1em}
\subsection{Related Work}
The work presented here is focused on the study of the demand-side management framework through noncooperative game theory. 
We build our results from \cite{Atzeni2012}, where the authors consider a day-ahead scheduling algorithm and the users optimize energy generation and storage profiles in a distributed manner.
Our model extends~\cite{Atzeni2012} by adding error terms to the expected consumption profiles (energy loads in the smart grid) and admitting variations in the estimated values with a bounded error. 
In addition, we analyze the effect that these error terms produce in the expected energy demand.
In order to do so, we extrapolate some techniques from reference~\cite{Cumanan2008}, where a robust algorithm for communications is developed.

Regarding a real-time energy market, several models have been proposed in the literature. 
For instance, the authors from reference \cite{Atzeni2012a} propose a real-time cost function with penalization terms in which the users minimize the expected monetary expenses. 
The energy production model assumes that the cost of energy is given by a quadratic function of the total demand, and therefore, the whole energy price of the smart grid is affected by the real-time variations.
In such setting, the energy price is predetermined one day ahead, but it does not account for the unpredictable production costs of real-time demand. 
The penalty functions should account for these extra expenses, but they are also predetermined in advance. 
If they are set too costly, they would affect greatly the monetary expenses of users in the real-time market. 
On the other hand, if the penalties are set too low, they would cause an economic loss to the energy supplier.
Our robust model, on the contrary, sets up an energy price that will allow the energy supplier to produce energy without incurring unplanned extra costs, and will allow users flexibility in their energy demand.

Another real-time model is introduced by reference \cite{Samadi2010}, which proposes an optimization problem in which users demand energy according to a personal utility function.
In this scenario prices are updated according to production costs, but demand is fully elastic, within some predetermined bounds, and the problem optimizes a social equilibrium point, rather than accounting for individual energy requirements.
Furthermore, the energy supplier may experience monetary losses since the pricing mechanism does not cover production costs. 
In contrast, our algorithm accommodates individual energy requirements and ensures that production costs are covered by the pricing mechanism.

The authors of reference \cite{Conejo2010} propose a robust scenario for a real-time market based on linear programming, which allocates energy load profiles within a time frame. 
Because of the uncertainty of prices in real-time, users decide when to consume energy according to some modeling parameters, price range estimates and energy requirements.
Although the algorithm is robust in the sense that it assumes a range of real-time energy prices, the model considers that prices are given by the energy supplier without taking into account production costs. 
Our robust model, on the other hand, establishes the prices one day ahead considering possible worst-case real-time variations, and takes into account production costs.

A relevant survey describing some of the last results and problems of the smart grid infrastructure can be found in~\cite{Saad2012}, which analyzes different contributions and open problems in the topic, where demand-side-management is one example of these. 
Another publication that models the interaction among resident users and distributed energy resources is the one introduced by~\cite{Tushar2014}, but it considers prefixed energy generation and storage load profiles, as opposed to our framework, which can adapt the profiles depending on the energy prices. 
Regarding user behavior, the authors of reference~\cite{Wang2014} propose a model that takes into account the subjective behavior of users, but still does not consider imprecise load predictions. 

Therefore, our contribution lies in considering a day-ahead energy market and minimizing the energy costs of users in a robust situation where energy demand errors are taken into account.
In Section~\ref{sec:model} we describe the demand-side model, the day-ahead cost function, the real-time market and briefly describe the day-ahead optimization process.
In Section~\ref{sec:worstcase} we introduce the robust game and analyze it's solution through an equivalent formulation.
We also present a distributed algorithm that solves the game and analyze it's convergence properties.
Finally, in Section~\ref{sec:simulations} we present simulations to illustrate the validity of our results.

\vspace{-0.5em}
\section{Smart Grid Model}
\label{sec:model}
In this section we introduce the smart grid model, the types of users in the demand-side of the network, the energy model, the pricing plan and the error terms. We consider a general framework with arbitrary convex regions and convex objectives.

\vspace{-1em}

\subsection{Demand-Side Model}
The smart grid topology can be divided in three parts, namely the supply-side, a central unit or regulation authority, and the demand-side. 
The supply-side involves the energy producers, as well as the  distribution network. 
The central unit coordinates the optimization process.
The demand-side entails the individual end-users as consumers, possibly with means to generate and/or store energy and/or vary their consumption profiles.
The demand-side may also include users with great impact on the smart grid, such as industries with high energy demands or opportunistic agents who participate as local energy suppliers in the network.
 
In this paper, the demand-side management focuses on calculating optimal policies for monetary savings on the consumers' side, while accounting for worst-case deviations from the consumption profiles due to unpredictable considerations in the demand-side. 
The optimization process involves all users together with the central unit, who exchange information for a more efficient use and distribution of energy. 
This communication is accomplished with smart meters, which are the devices in charge of the bidirectional transmission of data with the central unit and the optimization process. 

We assume that all end users indicate their intended amount of energy consumption one day ahead, so that energy prices can be calculated based on demand levels during the optimization process.
Then, the users can readjust their generation, storing and consumption profiles according to the variability of prices. 
This process is repeated iteratively until convergence. Algorithms that analyze this process have been studied in \cite{Atzeni2012} and \cite{Mohsenian-Rad2010}. 
However, due to errors in their estimations, the real consumption may vary from the announced values. 
If these variations are uncorrelated among users, sometimes the energy left unused by some consumers will be spent by others inducing that, on average, the supply-side will not require to adjust their generation rates. 
However, this expected result may not always be unconditionally true.
In addition, there is interest from the supply-side to provide energy for unexpected events that may require higher demands. 
Therefore, we present analytical results on identifying worst-case energy errors, for a given confidence interval.

We represent the set of users who participate in this process with $\mathcal{D}$, and additionally subdivide them into the set of active users $\mathcal{N}$ and the set of passive users $\mathcal{P}$. 
Active users have some means of generating energy on demand (from dispatchable sources), possess devices to store energy for future use or adapt their consumption profiles. 
On the contrary, passive users do not adjust their energy consumption profiles and their consumption is estimated from historical data.

The consumption values are determined one day ahead among all users, and are subdivided in time-slots $h\in H=\{1,\ldots,|H|\}$.
We define the \textit{estimated per-slot energy load profile} $l_n(h)$ that indicates the estimated energy usage of user $n\in \mathcal{D}$ at time-slot $h$. 
We also represent these variables as row vectors $\bm{l}_n=[l_n(1),\ldots,l_n(|H|)]\in\mathbb{R}^{|H|}$, $\forall n\in \mathcal{D}$.
Additionally, let $\bm{l}=[\bm{l}_1,\ldots,\bm{l}_{|\mathcal{D}|}]\in\mathbb{R}^{|H||D|}$ denote the global vector with the energy loads of all users.
Each individual term $\bm{l}_n$ includes the energy consumption of user $n$ and the energy contributions or added expenditures that may decrease or increase the energy load due to the user owning some device to generate or store energy, respectively.
Naturally, these profiles are variables to the optimization process and will vary according to the different strategies the user can select. 
It is satisfied that $l_n(h)>0$ if the energy flows from the grid to the user, and $l_n(h)<0$ otherwise, when the user is selling energy to the grid. 

We assume that the total load in the network--\textit{aggregate per-slot energy load}--at time slot $h$ denoted by $L(h)$ is positive:
\begin{IEEEeqnarray}{c}
	0<L(h) \triangleq \sum_{n\in \mathcal{D}}l_n(h)
	\vspace{-0.3em}
	\label{eq:total-load}
\end{IEEEeqnarray}
and it is calculated by the central unit by aggregating everybody's estimated load, including passive and active user's load.

Finally, we introduce the \textit{per-slot energy load error profile} $\delta_n(h)$ as the difference from the estimated \textit{per-slot energy load profile} $l_n(h)$ and the real value. 
This difference comes from the assumption that demand-side users do not know their consumption requirements with precision and they may deviate from their estimates. 
Additionally, active users may also deviate from their calculated generation and storing policies for unpredictable reasons such as malfunctioning, disconnections or lack of collaboration. 
For these reasons, introducing an error term makes the model more realistic. 
We represent these error terms both as row vectors $\bm{\delta}_n=[\delta_n(1),...,\delta_n(H)]$, $\forall n\in \mathcal{D}$ as the error profiles of user $n$ and as column vectors $\bm{\delta}(h)=[\delta_1(h),...,\delta_{|\mathcal{D}|}(h)]^T$, $\forall h\in H$.   
Note that the index $n$ or $h$ allows to easily distinguish which one we are referring to.
Additionally, let $\bm{\delta}=[\bm{\delta}_1,\ldots,\bm{\delta}_{|\mathcal{D}|}]\in\mathbb{R}^{|H||\mathcal{D}|}$ denote the global vector with the error terms of all users.


\subsection{Day-ahead Cost Model}
The purpose of our day-ahead model is to establish an estimate of the energy price that is robust against real-time energy demand variations. 
From a practical perspective, it determines a worst-case energy production cost taking into account expected user energy loads plus uncertainties, so that the energy supplier does not suffer monetary losses due to unplanned real-time demand.

The energy price in the network is determined by the \textit{grid cost function} $C_h$, which is fixed by the supply-authority and depends on the aggregate per-slot energy load $L(h)$ at each time-slot. 
We use a quadratic cost function, which is widely used in the literature (see e.g.~\cite{Atzeni2012,Mohsenian-Rad2010}):
\vspace{-0.5em}
\begin{IEEEeqnarray}{c}
	C_h(L(h),\bm{\delta}(h)) \triangleq K_h\Big(L(h)+\sum_{n\in \mathcal{D}}\delta_n(h)\Big)^2
\label{eq:grid_cost}
\vspace{-0.5em}
\end{IEEEeqnarray}
where the coefficients $\{K_h\}_{h=1}^{|H|}>0$ represent the energy cost at each time slot determined by the supply-side. 
In addition, the individual cost for user $n\in\mathcal{D}$ at time slot $h$ is its proportional part of the total cost, i.e., $C_h(L(h),\bm{\delta}(h))\cdot \Big((l_n(h)+\delta_n(h))/(L(h)+\sum_{n\in \mathcal{D}}\delta_n(h))\Big)$, so that its cumulative extension for all time-slots is given by:
\begin{equation}
	f_n(\bm{l},\bm{\delta})\triangleq\sum_{h\in H}K_h\Big( L(h)+\bm{1}^T\bm{\delta}(h)  \Big) (l_n(h)+\delta_n(h))+M(\bm{\delta}_n)
\label{eq:indiv_cost}
\vspace{-0.3em}
\end{equation}
where $\bm{l}_n$ is a convex function; $\bm{l}_n\in \Omega_{\bm{l}_n}$; all regions $\Omega_{\bm{l}_n}$ are convex, compact and independent among users; and $\bm{1}$ is a column vector of appropriate size.
Additionally, we have included a penalization term $M(\bm{\delta}_n)$ that includes management or distribution costs, which accounts for the extra expenses that the supply authority has to account for during the pricing process.
We separate these penalization terms from the energy prices due to the inherent unpredictable nature of the energy demand, which is an additional consideration with regards to the models in references~\cite{Atzeni2012,Mohsenian-Rad2010}.

Since we consider an arbitrary form of $\bm{l}_n$ and $\Omega_{\bm{l}_n}$, our robust proposal can be adapted to different demand-side models.
For instance, in the model from~\cite{Atzeni2012}, $\bm{l}_n$ has a linear expression which is a function of the generation and storing capabilities of its users. 
Likewise, the model from~\cite{Mohsenian-Rad2010} optimizes an energy user demand scheduling algorithm where $\bm{l}_n$ is linear and the feasible region is convex. 
We can, therefore, adapt a robust energy price that accounts for real-time uncertainties. 
And apart from the models \cite{Atzeni2012} and \cite{Mohsenian-Rad2010}, we can trivially assume a third one: one in which all users are passive, have no elastic demand at all, and simply do not optimize their energy load profiles.
In such case, our robust model can still be applied to determine worst-case estimation errors.

Finally, the error terms we have accounted for in~\eqref{eq:grid_cost} contribute to the user's monetary expense $f_n(\bm{l},\bm{\delta})$, since they affect the real-time energy price. 
In order to deal with these unknown demand terms, in Section \ref{sec:worstcase} we propose a worst-case analysis by solving a coupled min-max game.
This analysis will allow the supply authority to provide for excess energy demands and establish a robust energy price that accounts for unplanned demand.
In the next subsection we present how to integrate these worst-case error terms in a real-time scenario to obtain a robust model.
\vspace{-1em}

\subsection{Real-time Pricing Model}
\label{sub:real-time-pricing-model}
Once the robust price has been established one day ahead, if their demand remains within the worst-case estimation error limits, users are charged the amount of energy they demand in real time at the robust price. 
Outside of these limits, a penalty function is introduced to cover any extra expenses.
Specifically, our proposed real-time cost model is
\begin{equation}
	f_{n}^{\text{rt}}\left(\bm{l},\bm{\delta}^{\ast},\bm{l}^{\text{rt}}\right)=\sum_{h\in H}K_{h}\hat{L}_{r}(h)\big( l^{\text{rt}}_{n}(h)
	+ \Psi_{nh}(\bm{l},\bm{\delta}^{\ast},\bm{l}^{\text{rt}}) \big)+ M(\bm{\delta}_n) 
\label{eq:robust_real_time_cost}
\end{equation}
where $\bm{l}$ is the energy load profile obtained from the day-ahead optimization process, $\bm{\delta}^{\ast}$ is the worst-case estimation error (which is analyzed in Section~\ref{sec:worstcase}), and
\begin{equation}
	\hat{L}_{r}(h) = \sum_{m\in\mathcal{D}} l_m(h) + \delta^{\ast}_m(h)
	\label{eq:robust-price}
\end{equation}
is the robust total demanded energy plus worst-case estimation errors, calculated in the day-ahead optimization process. Variable $\bm{l}^\text{rt}$ refers to the actual energy load that users demand in real-time, and $\Psi_h$ is a penalty function that increases the cost whenever a user deviates from its specified load profile range; for example the following function is a real-time cost model inspired by~\cite{Atzeni2012a}:
\begin{IEEEeqnarray}{c}
\begin{IEEEeqnarraybox}[][c]{rCl}
	\Psi_{nh}(\bm{l},\bm{\delta}^{\ast},\bm{l}^{\text{rt}}) &=& \nu_h \big(l_n(h)-\delta^{\ast}_n(h)-l_n^{\text{rt}}(h)\big)^+  \\
	&&+ \upsilon_h \big(l_n^{\text{rt}}(h)-l_n(h)-\delta_n^{\ast}(h)\big)^+
	\label{eq:penalization-function3}
\end{IEEEeqnarraybox}
\end{IEEEeqnarray}
where $\nu_h$ and $\upsilon_h$ are scalar terms and $(\cdot)^+=\max(\cdot,0)$.
The penalty values discourage the return of unused energy during night or low demand times, and higher use of energy during day or high demand times.

Note that the penalty function proposed in~\eqref{eq:penalization-function3} differs from~\cite{Atzeni2012a} in that it does not charge any extra cost within the uncertainty margins. 
The reason is that our robust algorithm already takes into account the extra costs of unplanned demand from the production model. 
Note also that a user can still deviate further than the worst-case error estimates and, in such case, some penalty function should dissuade the user from these extremes; nevertheless, these deviations are limited with some probabilistic confidence.

Thus, our robust model acknowledges the cost of producing a certain amount of energy plus uncertainties, whereas current real-time models (like~\cite{Atzeni2012a}) only establish a surcharge for extra demand. Robust pricing models are of interest for energy suppliers as they take into account production costs under real-time price fluctuations. This is an important remark, because the real-time demand affects the whole energy price, and not only the extra demanded energy.

\vspace{-1em}
\subsection{Day-ahead Optimization Process}
We can now introduce the game $\gameA=\left\langle \Omega_{\bm{l}},\bm{f} \right\rangle$ that minimizes the user's monetary expense with fixed $\bm{\delta}$:
\begin{IEEEeqnarray*}{c'l'l}
	\gameA: & \left\{  \,
\begin{IEEEeqnarraybox}[][c]{r'l}
	 \min_{\bm{l}_n} & f_n(\bm{l}_n,\bm{l}_{-n},\bm{\delta})   \\
	 \text{s.t.} & \bm{l}_n\in \Omega_{\bm{l}_n},
\end{IEEEeqnarraybox} \right. & \forall n\in \mathcal{D}
\label{eq:game1}
\end{IEEEeqnarray*}
where $\bm{l}_{-n}=[\bm{l}_1,\ldots,\bm{l}_{l-1},\bm{l}_{l+1},\ldots,\bm{l}_{|\mathcal{D}|}]$ represents the rest of the other user's policies, which influence the user's decision. 
We consider the Nash Equilibrium (NE) as the solution concept of interest for this game, defined as a feasible point $\bm{l}^\ast=[\bm{l}_1^\ast,\cdots,\bm{l}_{|\mathcal{D}|}^\ast]$, such that it satisfies
\begin{IEEEeqnarray}{r'l}
	f_n(\bm{l}_n^\ast,\bm{l}_{-n}^\ast,\bm{\delta}) \leq f_n(\bm{l}_n,\bm{l}^\ast_{-n},\bm{\delta}) & \forall \bm{l}_n\in\Omega_{\bm{l}_n}
\end{IEEEeqnarray}
for every player $n\in\mathcal{D}$. 
In our framework, $\gameA$ has at least an equilibrium point, which can be reached through a convergent algorithm as established by the following theorem:

\begin{theorem}
\label{th:theorem-gameA}
The following holds for the game $\gameA$ with given and fixed error values $\bm{\delta}$:
\begin{enumerate}[label=\alph*)]
\item It has a nonempty, compact and convex solution set.
\item The individual cost values for the payoff functions of each player are equal for any NE solution of the game.
\item The game is monotone and NE solutions can be reached through algorithms of the kind proposed in \cite{Scutari2012}.
\end{enumerate}
\end{theorem}
\begin{proof}
See Appendix \ref{ap:proof-th1}.
\end{proof}

Convergent algorithms to reach such NE solutions are described in \cite{Scutari2012}, which depend specifically on the convex structure of the game.
In particular, they require a single-loop best response algorithm if the game is strongly monotone, or a double-loop best-response algorithm if it is just monotone (with an added proximity term). 
The specific steps, which are similar to those described in \cite[Alg.1]{Atzeni2012} for their particular model, but including fixed error terms, are as follows. 
First, users determine their respective~$\bm{l}_n$ based on a day-ahead optimization process (considering~$\bm{\delta}$ fixed), where every user knows their energy consumption requirements and energy costs, and calculate their best strategies depending on the energy prices. 
Then, all users communicate their intended energy load profile~$\bm{l}_n$ to the central unit and energy prices are recalculated. 
This procedure is repeated iteratively until all users converge to an NE. 

\section{Worst-Case Analysis of the Energy Load Profiles}
\label{sec:worstcase}
In the formulation of $\gameA$ we implicitly assumed that $\bm{\delta}$ is given or known a priori, or otherwise is set to zero. However, these error profiles are in fact unknown random terms that we added in order to correct any deviation from the real performance.
In this section, these error terms will be treated as optimization variables, and analyzed from a worst-case performance perspective.
This allows the regulation authority to jointly consider both the price and management/availability costs of energy.
With this idea in mind, it is realistic to consider bounded error terms, represented by
\begin{IEEEeqnarray}{r'l}
	\Delta_h \triangleq \{\bm{\delta}(h)\in\mathbb{R}^{|\mathcal{D}|}\ |\ \|\bm{\delta}(h)\|^2_2\leq\alpha(h)\} & \forall h\in H
\label{eq:quad_constraint}
\end{IEEEeqnarray}
where $\alpha(h)$ are prefixed in advance and depend on the time-slot $h$ due to more or less confidence in the regulator predictions.
In practice, these values can be inferred by the regulation authority, based on historic data from consumers and expected variations.
Quadratic constraints on the error terms are common in the literature, see e.g. \cite{Cumanan2008,Wang2011}.

In order to propose this error model, we have assumed two conditions that are satisfied generally in practice: \emph{i)}~independence of the error terms among users, and \emph{ii)}, no knowledge of the specific probabilistic error profiles of users. 
Property \emph{i)} is satisfied if we assume that users alter their predicted energy profiles independently (e.g. a user decides to charge an electric vehicle for an unexpected event). 
The energy profiles of different users can be correlated, but what we assume to be independent are the deviation errors from these users. 
Property \emph{ii)} implies that if we do not have any knowledge of the error distribution of a specific user, then it is sensible to substitute such profile with  the error distribution of an average user profile. 
By the central limit theorem such average user profile will approximate a Gaussian distribution. 
Then, the error contribution of all users can be modeled as a Gaussian distribution of zero mean, and variance the sum of all user variances. Specifically, equation~\eqref{eq:quad_constraint} represents an upper bound to the overall network error variance.

In the following, we analyze a scenario that considers that every player has to account individually for their worst-case error terms and is charged accordingly. The model assumes that each player is responsible for some maximum error profiles in joint consideration with the other players. 

Before the analysis, we bring up a useful result that guarantees global optimality for a maximization problem with quadratic convex objectives and constraints.
\begin{theorem}
\label{th:globalmax}
(from \cite[Th.3]{Hiriart-Urruty2001}). Given the problem
\begin{IEEEeqnarray}{l}
\begin{IEEEeqnarraybox}[][c]{r'l}
\max_{x\in\mathbb{R}^n} &\frac{1}{2}x^TAx+b^Tx+c \\
	\text{s.t.} & \frac{1}{2}x^TQx+dx+e\leq 0
\end{IEEEeqnarraybox}
\label{th:square-standard-prob}
\end{IEEEeqnarray}
where $A\neq 0,Q \in \mathbb{R}^{n \times n}$ are positive semidefinite matrices,
$b,d\in \mathbb{R}^n$ and $c,e\in \mathbb{R}$, and Slater's condition is satisfied: then $\bar{x}$ is a global maximizer iff there exists $\bar{\lambda}\geq 0$ satisfying
\begin{IEEEeqnarray}{l}
		A\bar{x}+b=\bar{\lambda}(Q\bar{x}+d)    \label{eq:kkt_cond}  \\
		-A+\bar{\lambda}Q \quad \text{is positive semidefinite}.   \label{eq:posdef_cond}
\end{IEEEeqnarray}
\end{theorem}
Note that \eqref{eq:kkt_cond} is a KKT necessary condition, plus a second requirement \eqref{eq:posdef_cond} that will guarantee global maximum.

\vspace{-1em}
\subsection{Worst-Case Min-max Game Formulation} 
\label{sec:worstcase-game}
In this section we consider a worst-case situation where the error terms are coupled among all of the users.
We analyze which are the user's best response strategies and how to calculate the worst-case error terms.
In particular, this transforms $\gameA$ into a game where all users have to solve a min-max cost function, with coupled interactions and constraints.
The new game formulation $\gameB$, where each user $n$ has to determine variables~$\bm{l}_n$ and $\bm{\delta}_n$, is given by
\begin{IEEEeqnarray*}{c'l'l}
	\gameB: & \left\{  \,
\begin{IEEEeqnarraybox}[][c]{r'l}
	\min_{\bm{l}_n} & \max_{\bm{\delta}_n} f_n(\bm{l}_n,\bm{l}_{-n},\bm{\delta}_n,,\bm{\delta}_{-n})  \\
	\text{s.t.} 
			    & \bm{\delta}(h)\in \Delta_h,\quad \forall h\in H  \vspace{-0.5em}\\
			    & \bm{l}_n\in \Omega_{\bm{l}_n}   
\end{IEEEeqnarraybox}  \right. & \forall n\in\mathcal{D},
\end{IEEEeqnarray*}
where region $\Omega_{\bm{l}_n}$ is convex, compact and nonempty. 
The min-max of game $\gameB$ for every user $n\in \mathcal{D}$ represents a game with a convex objective function in variables $\bm{l}_n$ and convex regions. 
Note that for the specific case that $M(\bm{\delta}_n)$ is quadratic and convex in variable $\bm{\delta}_n$, then the maximization problem falls into the category to readily use Theorem \ref{th:globalmax}, since all equations are quadratic.
More generally, we can provide the following theorem for any continuous function $M(\bm{\delta}_n)$:

\begin{theorem}
\label{th-NE-existence-gameB}
The following holds for game $\gameB$ with min-max objectives:
\begin{enumerate}[label=\alph*)]
\item It has a nonempty, compact and convex solution set.
\item It is monotone in variables $\bm{l}_n$, and NE can be reached through algorithms of the kind proposed in \cite{Scutari2012}.
\end{enumerate}
\end{theorem}
\begin{proof}
The max operation inside the objective of each user in game $\gameB$ is performed over a convex and compact set and, therefore, it is well defined (it always has a solution).
The maximum function in~$\gameB$ preserves convexity of the objective in variable $\bm{l}_n$.
Furthermore, if the point that maximizes the objective is unique, then the objective is differentiable in $\bm{l}_n$ with gradient $\nabla_{\bm{l}_n}f_n(\bm{l}_n,\bm{l}_{-n},\bm{\delta^\ast}_n,\bm{\delta}_{-n})$, and where $\bm{\delta}_n^\ast=\arg\max_{\bm{\delta}_n}f_n(\bm{l},\bm{\delta}_n,\bm{\delta}_{-n})$ denotes the maximal point.
This property is guaranteed by Danskin's Theorem \cite[Sec.~B.5]{Bertsekas1999}.
With this analysis in mind, the required assumptions from~\cite[Th.4.1.a]{Scutari2012} are satisfied and an NE exists.
Furthermore, the Jacobian analysis performed for Theorem \ref{th:theorem-gameA} (see Appendix~\ref{ap:proof-th1}) remains the same in $\gameB$, with the exception of substituting the fixed $\bm{\delta}_n$ with the optimal solution of the max problem $\bm{\delta}_n^\ast$.
Likewise, we can claim that the game is monotone and, therefore, that the solution set is also convex \cite[Th.4.1.b]{Scutari2012}.
By Danskin's theorem, if the maximum is not a unique point then the gradient would become a subdifferential, and slight variations would have to be considered that include the use of multifunctions \cite{Facchinei2010}.
\end{proof}

Game $\gameB$ falls into a rather novel category of min-max problems recently analyzed in \cite{Facchinei2014}.
Such reference proposes an alternative formulation for min-max games, where extra players are added to solve the maximization parts to decouple the min-max form.
Moreover, reference~\cite{Facchinei2014} proves equivalence of NE for both formulations, and proposes solutions under monotonicity and other requirements.
However, our problem $\gameB$ falls into a more general form with individual max objectives (rather than a common one), where the monotonicity property does not hold.
Therefore, the approach from~\cite{Facchinei2014} becomes unsuitable for distributed algorithms.
Furthermore, $\gameB$ incorporates global constraints which are not considered in the previous reference.
For this reason, we have to develop a novel approach for this particular problem.

In order to analyze and solve the coupled maximization problems for all players, we first formulate an equivalent game $\gameBt$, which is more tractable, together with a nonlinear complementarity problem (NCP) indicated by~\eqref{eq:wholeworst-tilde}.
The NCP is required to establish the equivalence together with other assumptions on the dual variables $\lambda_h$.
We have also particularized the cost term to $M(\bm{\delta}_n)=\betamm \Vert\bm{\delta}_n\Vert^2_2$ with $\betamm\geq 0$, which is motivated to penalize the user's local deviations.
Furthermore, this particular choice will allow us to obtain tractable results for the game and the NCP:
\begin{IEEEeqnarray}{l'l}
\gameBt:  \left\{  \,
\begin{IEEEeqnarraybox}[][c]{r'l}
	\min_{\bm{l}_n} & \max_{\bm{\delta}_n}\: p_n(\bm{l}_n,\bm{l}_{-n},\bm{\delta}_n,\bm{\delta}_{-n},\bm{\lambda}) \\
	\text{s.t.} & \bm{l}_n\in \Omega_{\bm{l}_n}   
\IEEEstrut
\end{IEEEeqnarraybox} \right. \quad \forall n\in\mathcal{D} \IEEEnonumber \\
0 < K_h+\betamm\leq \lambda_h \perp -\Big(\|\bm{\delta}(h)\|^2_2-\alpha(h)\Big) \geq 0,\;\forall h
\label{eq:wholeworst-tilde}
\end{IEEEeqnarray}
where $\bm{\lambda}=[\lambda_1,\ldots,\lambda_{|H|}]^T$, $a\perp b$ indicates $a^T b=0$, and
\begin{multline}
	p_n(\bm{l}_n,\bm{l}_{-n},\bm{\delta}_n,\bm{\delta}_{-n},\bm{\lambda}) = 
		\sum_{h\in H} (\betamm - \lambda_h) \delta_n(h)^2   \\
		+\sum_{h\in H} K_h\Big(L(h)+\sum_{k\in \mathcal{D}}\delta_k(h)\Big)(l_n(h)+\delta_n(h)).
\label{eq:minmax-game-pn}
\end{multline}
Next lemma shows the equivalence between $\gameB$ and $\gameBt$.
\begin{lemma}
\label{lemma:game-equivalence}
Games $\gameB$ and $\gameBt$ together with the NCP specified by equation \eqref{eq:wholeworst-tilde} present the same NE solutions.
\end{lemma}
\begin{proof}
The KKT system of equations of $\gameB$ and those of game $\gameBt$ jointly with the NCP are the same, and therefore, any NE candidate solution can be derived from either formulation. 
A maximum point in both problems exists since all regions are compact. Inequality $\lambda_h\geq K_h+\betamm$ guarantees that such point is maximum because Theorem~\ref{th:globalmax} is satisfied for every user in $\gameB$ and, likewise, is maximum in $\gameBt$ because it  makes the objective concave in $\bm{\delta}_n$. 
Therefore, the NE coincide.
\end{proof}

Note that, in this equivalent formulation of $\gameBt$ we have focused on the special case where all $\lambda_h$ are equal for all players for every~$h\in H$. 
The equilibrium points of this formulation are normally referred to as ``variational" or ``normalized" solutions~\cite{Facchinei2010a}, and the practical implication is that they establish a common price on a resource for all agents. 
These solutions also retain some stability properties compared to other equilibrium points where the dual variables are different among users, see~\cite{Facchinei2010a} for further details.

One benefit from dealing with the new formulation $\gameBt$ is that the {\em strong max-min property} is directly fulfilled if the objective is {\em convex-concave}, as opposed to $\gameB$ where it remains {\em convex-convex} (see \cite[Ex.3.14b]{Boyd2004} for further details). 
This implies that the min-max can be changed into a max-min and the solution is not altered, i.e., equivalence of both problems holds. 
We see this is true as long as $\lambda_h\geq K_h+\betamm$ because the objective is then concave in $\bm{\delta}_n$. This allows us to analyze the monotonicity of the new game setting~$\gameBt$ in these variables. Let us present a lemma that guarantees the existence of common $\lambda_h$'s satisfying the previous requirements:
\begin{lemma}
\label{lemma:lower-bound}
There always exists $\lambda_h$ satisfying \eqref{eq:wholeworst-tilde} (i.e., making $\gameBt$ concave in $\bm{\delta}_n$) that is dual optimal in $\gameB$. A lower bound to this value is:
$\lambda_h > \betamm+K_h$ for all $h\in H$.
\end{lemma}
\begin{proof}
See Appendix \ref{proof-theorem-monotoneBt}.
\end{proof}
Now, we can establish the following result.
\begin{theorem}
\label{th:monotone-gameBt}
The following holds:
\begin{enumerate}[label=\alph*)]
\item Game $\gameBt$ has a nonempty and compact solution set.
\item A variational solution of game $\gameB$ always exists.
\item If additionally $\lambda_h > \frac{1}{2}K_h(|D|+1)+\betamm$, then the game $\gameBt$ is strongly monotone in variables $\bm{\delta}_n$ for all $n\in \mathcal{D}$, it has a unique NE solution, and it can be reached by a distributed asynchronous best response algorithm of the kind proposed by~\cite[Alg.4.1]{Scutari2012}. \label{th4-c}
\end{enumerate}
\end{theorem}
\begin{proof}
See Appendix \ref{proof-theorem-monotoneBt}
\end{proof}

In order to solve $\gameBt$ and the NCP described by~\eqref{eq:wholeworst-tilde}, we need to study the inner game formed from the coupled maximization problems in $\bm{\delta}$. In order to do that, we assume that $\bm{l}$ is fixed. Specifically, we analyze the following game:
\begin{IEEEeqnarray}{r'l} \left\{
\begin{IEEEeqnarraybox}[][c]{r'l}
	\max_{\bm{\delta}_n\in\mathbb{R}^{|H|}} & p_n(\bm{l},\bm{\delta}_n,\bm{\delta}_{-n},\bm{\lambda}) 
\end{IEEEeqnarraybox} \right. &  \forall n\in \mathcal{D}
\label{eq:inner-game}
\end{IEEEeqnarray}
where~\eqref{eq:inner-game} is solved jointly with the NCP~\eqref{eq:wholeworst-tilde}.
First, note that all users' objectives can be separated in $|H|$ independent problems.
This is possible because the objective function $p_n$ is expressed as a sum over $h\in H$ and~\eqref{eq:wholeworst-tilde} is expressed separately for every $h\in H$.

Each individual problem has a quadratic form and, hence, Theorem~\ref{th:globalmax} can be applied to each of them. 
We can transform the model given by~\eqref{eq:wholeworst-tilde} and \eqref{eq:inner-game} into a form resembling~\eqref{th:square-standard-prob}, and identify terms $A=K_h+\betamm$ and $Q=1$. 
We conclude that (\ref{eq:posdef_cond}) is satisfied if $\lambda_h\geq K_h+\betamm$. Finally, we obtain the first order necessary conditions from~\eqref{eq:inner-game} for all $n\in\mathcal{D}$:
\begin{equation*}
	 K_h\Big(l_n(h)+L(h)+\sum_{k\neq n}\delta_k(h) \Big)  
 	 +2\Big(K_h+\betamm-\lambda_h\Big) \delta_n(h) =0,
\label{eq:set-linear-centralized}
\end{equation*}
and form a system of equations together with~\eqref{eq:wholeworst-tilde}.
By separating variables and solving, we get
\begin{equation} 
\label{eq:solgame_norm2}
	\delta_n(h)=\tfrac{\sqrt{\alpha(h)}\left(l_n(h)+L(h)+\sum_{k\neq n}\delta_k(h)\right)}{\sqrt{\sum_n(l_n(h)+L(h)+\sum_{k\neq n}\delta_k(h))^2}}, 
	\forall \:n\in \mathcal{D} 
\end{equation}
\begin{equation}
	\lambda_h = K_h +\betamm+ K_h\tfrac{\sqrt{\sum_n(l_n(h)+L(h)+\sum_{k\neq n}\delta_k(h))^2}}{2\sqrt{\alpha(h)}}
\label{eq:lambda-opt}
\end{equation}
where we have chosen positive sign to satisfy $\lambda_h\geq K_h+\betamm$. 

Since the solution to the fixed point equations is derived from conditions \eqref{eq:kkt_cond}--\eqref{eq:posdef_cond}, they are, therefore, global maximum of the individual objective functions. 
As no player can further maximize their objectives given the other players strategies, they form an NE. 
If the game is strongly monotone (Th. \ref{th:monotone-gameBt}c is satisfied), then asynchronous fixed-point iterations over these equations will converge to the global maximum. 
If the game is not strongly monotone, then the set of optimal equations given by~\eqref{eq:solgame_norm2}--\eqref{eq:lambda-opt} have to be solved in an alternative fashion. 
We propose a fixed-point iterative mapping that results from calculating equation \eqref{eq:solgame_norm2} repeatedly, plus a projection operation. 
The mapping in vectorial form becomes:
\begin{IEEEeqnarray}{c}
	T_h(\bm{\delta}(h),\bm{a}_h)=\Pi_{X_h}\bigg( \sqrt{\alpha(h)}\frac{\bm{a}_h+A\bm{\delta}(h)}{\Vert\bm{a}_h+A\bm{\delta}(h)\Vert } \bigg)
\label{eq:minmax-fixedpoint-map}
\end{IEEEeqnarray}
where $\bm{a}_h = L(h)+\bm{l}(h)$, $A = \bm{1}\bm{1}^T-\bm{I}$ with size $|\mathcal{D}|\times|\mathcal{D}|$ and $\Pi_{X_h}(\cdot)$ represents the euclidean projection onto region~$X_h$. 
Such region $X_h$ is defined as
\begin{equation}
	X_h\triangleq \Big\{ \bm{\delta}(h)\in \mathbb{R}_+^{|\mathcal{D}|}\ \big|
	\   \bm{1}^T \bm{\delta}(h)+ \tfrac{1}{|D|-1}\bm{1}^T\bm{a}_h -\sqrt{\alpha(h)|D|}\geq 0 \Big\}
\label{eq:set-X}
\end{equation}
where $\mathbb{R}_+^{|\mathcal{D}|}$ refers to the nonnegative quadrant, so that $X_h$ corresponds to the upper halfspace region limited by the hyperplane in the definition.

Next we give the following properties of the mapping $T_h$:
\begin{theorem}
Given the mapping $T_h(\bm{\delta}(h),\bm{a}_h)$ defined in \eqref{eq:minmax-fixedpoint-map} where $\bm{a}_h$ is fixed, then it follows: 
\begin{enumerate}[label=\alph*)]
	\item $T_h:X_h\rightarrow X_h$ is a self-map for every $h\in H$.
	\item $T_h$ is a contraction mapping.
	\item The update rule $\bm{\delta}(h)^{k+1} = T_h(\bm{\delta}(h)^{k},\bm{a}_h)$
		converges to a unique point $\bm{\delta}(h)^\ast\in X_h$ that solves the fixed point equation: 
		\begin{equation*}
			\bm{\delta}^{\ast}(h) = T_h(\bm{\delta}^{\ast}(h),\bm{a}_h).
		\end{equation*}
	\item The point $\bm{\delta}^{\ast}(h)$ is an NE of game $\gameBt$, with $\bm{\lambda}$ satisfying~\eqref{eq:wholeworst-tilde} and fixed 
		strategies~$\bm{l}(h)$.
\end{enumerate}
\label{th:minmax-fixpoint-mapping}
\end{theorem}
\begin{proof}
See Appendix \ref{th:minmax-fixpoint-proof}
\end{proof}

Because of Theorem~\ref{th:minmax-fixpoint-mapping} and the equivalence principle of Lemma~\ref{lemma:game-equivalence}, the fixed-point solution of the mapping also maximizes the objectives of game~$\gameB$.
Thus, the mapping provides an iterative method to find a solution of the game, regardless of its monotonicity properties.

Region $X_h$ in~\eqref{eq:set-X} becomes a technical adjustment that allows the operator to be a contraction self-mapping.
It has an involved interpretation, but in simple terms it limits the search space of the fixed-point to the upper region of the hyperplane.
In particular, in such upper halfspace the mapping is always a contraction, while in the lower halfspace such property is not always satisfied.
Therefore, a fixed point in $X_h$ always exists, and a recursive algorithm will converge to such point.

Finally, we can explicitly give an expression of the projection operation into the halfspace $X_h$, which is given by
\begin{IEEEeqnarray*}{c}
\Pi_{X_h} (x) \triangleq \left\{
\begin{IEEEeqnarraybox}[][c]{l}
	x  \hfill \text{if } x\in X_h\; \\
	x + \Big(\sqrt{\alpha(h)}-\big(A^{-1}\bm{a}_h+x\big)^T \tfrac{\bm{1}}{\Vert \bm{1}\Vert} \Big)
		\tfrac{\bm{1}}{\Vert \bm{1}\Vert} \quad \\
	 	\hfill \text{if } x\notin X_h,
\end{IEEEeqnarraybox}
\right.
\end{IEEEeqnarray*}
which only involves linear operations.

\subsection{Distributed Robust Algorithm}
Distributed implementations are beneficial for the smart grid because of their computation scalability.
Therefore, we propose the distributed Algorithm~\ref{alg:worst-case-game}.

The algorithm solves $\gameB$, and because of the monotonicity shown in Theorem~\ref{th-NE-existence-gameB} in variable $\bm{l}$, convergence is guaranteed. 
The proximal term included in step (S.3) is necessary if $\gameB$ is only monotone, and $\tau$ needs to be sufficiently large and satisfy that the game becomes strongly monotone, as indicated in~\cite[Cor.4.1]{Scutari2012}.
Any optimization toolbox would suffice to solve the strongly convex problem indicated in (S.3), which has a unique solution.
In step (S.4.), we solve the inner max game $\gameBt$ among all players with fixed components $\bm{l}^{i+1}_n$. In order to do so, we iteratively calculate $T(\bm{\delta}(h)^{k},\bm{a}_h)$ until convergence, which is guaranteed by Theorem~\ref{th:minmax-fixpoint-mapping}.
The main outer loop is also repeated until convergence.

Regarding the information exchange costs, users solve their local optimization problems, and then communicate their intended energy load profiles to the central unit. The central unit aggregates all energy loads and computes the individual worst-case error terms and the robust energy price defined in~\eqref{eq:robust-price}. Then, the robust energy price and the individual worst-case error estimates are communicated back to the users (privacy of individual values is preserved). Finally, the users re-optimize their new energy load profiles.

It seems sensible that the inner loop described on step~4 of Algorithm~\ref{alg:worst-case-game} is performed by the central unit. It involves simple operations described by the mapping $T_h(\bm{\delta}(h),\bm{a}_h)$, which scales linearly with the number of users $|\mathcal{D}|$, and requires no communication in order to find a solution. An alternative computation of the error terms in a distributed fashion is also possible, but it would increase the required communication.
\begin{algorithm}
\caption{Distributed robust algorithm for game $\gameB$}
\begin{algorithmic}
\item (S.1) Given $K_h$ for all $h\in H$. Set iteration index $i=0$. \\
		\hspace{2em}Set $\tau > 0$. 
		Initialize $\bm{\delta}^0(h)=\tfrac{\sqrt{\alpha(h)}}{\sqrt{|\mathcal{D}|}}\bm{1}$.
\item (S.2) If a termination criterion is satisfied, STOP.
\item (S.3) Solve for every $n\in\mathcal{N}$
\begin{equation*}
	\bm{l}_n^{i+1} = \arg\min_{\bm{l}_n\in\Omega_{\bm{l}_n}} f(\bm{l}_n,\bm{l}_{-n}^i,\bm{\delta}^i) 
	+\frac{\tau}{2}(\bm{l}_n-\bar{\bm{l}}_n^i)^2
\end{equation*}
\item (S.4) Iterate from $k=0$ until convergence
\begin{equation}
	\hat{\bm{\delta}}^{k+1}(h)=T(\hat{\bm{\delta}}^{k}(h),\bm{a}^{i+1}_h),\qquad \forall h\in H
\end{equation}
\hspace{2em} where $\bm{a}_h^{i+1}=\sum_{n\in \mathcal{D}}l_n^{i+1}(h)+\bm{l}^{i+1}(h)$.\\
\hspace{2em} Set the fixed-point solution: $\bm{\delta}^{i+1} \leftarrow \hat{\bm{\delta}}^{k+1}$.
\item (S.5) If an NE has been reached, update centroids $\bar{\bm{l}}_n^i = \bm{l}_n^{i}$.
\item (S.6) Set $i\leftarrow i+1$. Go to (S.2).
\end{algorithmic}
\label{alg:worst-case-game}
\end{algorithm}

\vspace{-1em}

\section{Simulation Results}
\label{sec:simulations}
In this section, we present simulation results from our robust analysis.
First, in Subsection~\ref{sub:simu-realtime} we present a real-time cost comparison with a non-robust model to illustrate that the robust model can effectively reduce the monetary costs.
Secondly, in Subsection~\ref{sub:simu-naive-comparison}, we compare our algorithm with a model formed by naive users who do not take into account the worst-case error terms.
We show that the monetary loses in this case are not negligible, and emphasize the importance of optimizing the energy loads taking into account the error terms.
In addition, we include convergence results.

In the following simulations, we used Algorithm~\ref{alg:worst-case-game} to determine $\bm{\delta}_n$ and $\bm{l}_n$ applied to a scenario similar to the one proposed in~\cite{Atzeni2012}.
In this kind of scenario, users decide upon local energy generation and storage strategies.
The total number of users $|\mathcal{D}|$ will vary between simulations, but the proportion of active users remains fixed in all of them, where $|\mathcal{N}|=\lfloor|\mathcal{D}|/2\rfloor$. 
From the active users set $\mathcal{N}$, the number of users that generate and store energy is $GS=|\mathcal{N}|-2\lfloor|\mathcal{N}|/3\rfloor$, the number of users who only generate is $G=\lfloor|\mathcal{N}|/3\rfloor$ and the number of users who only store energy is $S=\lfloor|\mathcal{N}|/3\rfloor$. 
The day is divided in $|H|=24$ time slots. 
The average consumption per user is of $4.5\:\unit{kWh}$ with cost coefficients $K_{\rm{day}}=1.5K_{\rm{night}}$, where we assumed there is lower price at times $h_{\rm{night}}=[1,...,|H|/3]$ and higher price at time slots $h_{\rm{day}}=[|H|/3+1,...,|H|]$.
These values are determined so that the initial average price for the day without any demand-side management is of $0.1412\:\unit{\pounds/kWh}$ (same as in~\cite{Atzeni2012}).

Other parameters involving active users with generation capabilities, $GS$ and $G$, are: $g_n^{\rm{(max)}}=0.4\:\unit{kWh/h}$, which is the maximum energy production capability per hour by the user's owned infrastructure; and $\gamma_n^{\rm{(max)}}=0.8g_n^{\rm{(\max)}}$, which is the maximum energy production capability per day. 
For active users owning storing devices, $GS$ and $S$, we define the following variables: the hourly batteries leakage rate $\alpha_n=\sqrt[24]{0.9}$, the charge efficiency $\upsilon_n^+=0.9$, the efficiency at discharge $\upsilon_n^-=1.1$, the capacity of the storage device $c_n=4\:\unit{kWh}$, the maximum charge rate per hour $s_n^{(max)}=0.125c_n$, the charge level at the beginning of the day $q_n(0)=0.25c_n$ and the difference of charge at the next day $\epsilon_n=0$.
All these variables are further explained in reference~\cite{Atzeni2012}. In essence, the parameters presented form a convex model of the problem that can be solved using convex optimization solvers.

The energy demand profiles per user have been simulated using a typical day profile similar to the one used in~\cite{Atzeni2012} plus some random Gaussian noise that adds some variance to the users.
The $\alpha(h)$ terms have been fixed to be the $10\%$ of the total energy demand per time slot $h\in H$.

\vspace{-1em}
\subsection{Real-time Cost Comparison}
\label{sub:simu-realtime}
In this subsection, we present a comparison of our robust model with the non-robust model from~\cite{Atzeni2012a}, when users deviate from their day-ahead optimized energy load profiles. 
Specifically, we present the average cost function over 100 independent simulations, where the real time error deviations are given by white Gaussian processes with zero mean and variance $\alpha(h)/|\mathcal{D}|$ for each user, and where the number of users $|\mathcal{D}|$ ranges from 100 to 2000. 
The cost function we use to evaluate the performance in real-time of Algorithm~\ref{alg:worst-case-game} is described in~\eqref{eq:robust_real_time_cost}.
Likewise, the cost function for user $n$ we compare with, denoted as `non-robust', has the following form:
\begin{IEEEeqnarray}{rCl}
	f_{n}^{\text{nrt}}\left(\bm{l},\bm{\delta}^{\ast},\bm{l}^{\text{rt}}\right)&=&\sum_{h\in H}K_{h}\hat{L}_{wc}(h)\Big( l^{\text{rt}}_{n}(h)
	+ \nu_h \big(l_n(h)-l_n^{\text{rt}}(h)\big)^+  \IEEEnonumber \\
	& & + \upsilon_h \big(l_n^{\text{rt}}(h)-l_n(h)\big)^+ \Big),
\label{eq:pen-function-non-robust}
\end{IEEEeqnarray}
which has no error terms in contrast to~\eqref{eq:robust_real_time_cost}.
Note also that we are representing the average cost of all users in the network.
Finally, values $\nu_h$ and $\upsilon_h$ are set to $\nu_h=0.2\kappa$ for $h\in[0,\ldots,8]$, $\nu_h=0.8\kappa$ for $h\in [9,\ldots,24]$ and $\upsilon_h=(\kappa-\nu_h)$, where $\kappa = \sqrt{{|\mathcal{D}|}/{\alpha(h)}}$. 
The specific simulation parameters $\nu_h$ and $\upsilon_h$ are obtained from~\cite{Atzeni2012a}, and parameter $\kappa$ is introduced to incorporate the uncertainty of the real-time consumption in the penalty terms. 
Note that the value of $\kappa$ is the inverse of the deviation of every user's real-time error distribution.

The average results of the cost models are presented in Figure~{\ref{fig:real-time}.
The term $M(\bm{\delta}_n)$ is not present in~\eqref{eq:pen-function-non-robust}, so we set $M(\bm{\delta}_n)=0$ in~\eqref{eq:robust_real_time_cost} for a fair comparison.
Commenting on these results, we observe that the cost for the users is significantly lower than with a penalty based model as~\eqref{eq:pen-function-non-robust}.
Therefore, this model effectively reduces the monetary expenses of all users in the real-time market, while at the same time provides a reliable production cost estimate to the energy supplier.
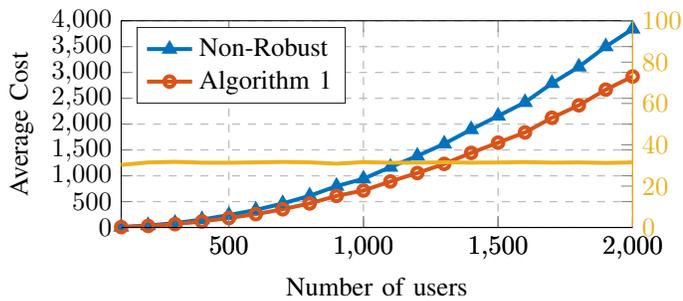
\begin{figure}
    \centering
    \input{realtime_cost.tikz}
	\caption{Real time averaged cost functions over 100 simulations comparing robust function~\eqref{eq:robust_real_time_cost} and~\eqref{eq:pen-function-non-robust}. The left axis shows the monetary cost of all users and the right axis shows (in \%) the mean relative gain when using the robust algorithm.}
	\label{fig:real-time}
\vspace{-1em}
\end{figure}

\vspace{-1em}
\subsection{Comparison of Algorithm~\ref{alg:worst-case-game} with naive users}
\label{sub:simu-naive-comparison}
In order to further support our results, we include the performance of our robust algorithm and compare it to a scenario where users do not consider the worst-case error terms in their energy load profiles. 
In \reffig~\ref{fig:minmax}, we plot the monetary cost of all users in different sizes of the smart grid. 
The number of users in this plot varies from 100 to 2000 users, and the total number of time slots is $|H|=24$. 
The curve we refer to as ``naive users" is the solution to the game without accounting for any error terms (i.e., $\bm{\delta}=0$), but with added worst-case calculations in the price determined by the energy supplier.
On the other hand, the robust curve represents the total energy cost when using Algorithm~\ref{alg:worst-case-game}, which does take into account the worst-case error terms. 
The average relative gain over all users when using the robust algorithm is around the 7\%, which remains almost constant with the number of users; these savings values are represented in the right hand axis of the figure for different number of users.
This result illustrates that users should adjust their energy load profiles after taking the worst-case error terms into account.
For these simulations we fixed the term $\betamm=0.001$.

In Figure~\ref{fig:conv-alg1-1000users}, we plot the convergence rate of Algorithm 1, which depicts the averaged convergence rate particularized for 1000 users.
The convergence rate is very similar to the results presented in earlier approaches (see, e.g., as \cite[Alg.1]{Atzeni2012}).
Our algorithm has the same algorithmic steps except for the inner loop involved in step 4, which we didn't observe to impact the computation time. 
In \reffig~\ref{fig:convergence}, we plot the convergence speed of the contraction mapping from step (S.4) in Algorithm~\ref{alg:worst-case-game}.
The x-axis shows the number of iterations and the y-axis shows the magnitude $\sum_{h\in H}\Vert \bm{\delta}^{k+1}(h)-\bm{\delta}^k(h)\Vert$.
The stopping criteria was set to~$10^{-8}$.
Recall that convergence is guaranteed since the computation in (S.4) is a contraction mapping.

The proposed algorithm was implemented in Matlab R2015a using an Intel i5-2500 CPU @ 3.30GHz.
The user's optimization problems were run in sequential order (although in practice they can run in parallel). 
The local convex optimization problems have unique optimal points, so any optimization solver will produce similar results as ours.
In particular, we used the \emph{fmincon} function from Matlab. 
The whole simulation took about 61.38 seconds for 1000 users.

As a conclusion, these simulations show consistent results in which there is an actual benefit of around the 7\% when using a robust algorithm vs. a naive one (i.e., that does not take into account the error terms for the optimization process), while the extra computation required is minimal.
\begin{figure}
    \centering
    \input{mm2000u24h.tikz}
	\caption{Energy cost of all players vs. number of users. The left axis shows the monetary cost and the right axis shows the mean relative gain (in \%) when using the robust algorithm.}
	\label{fig:minmax}
	\vspace{-1em}
\end{figure}
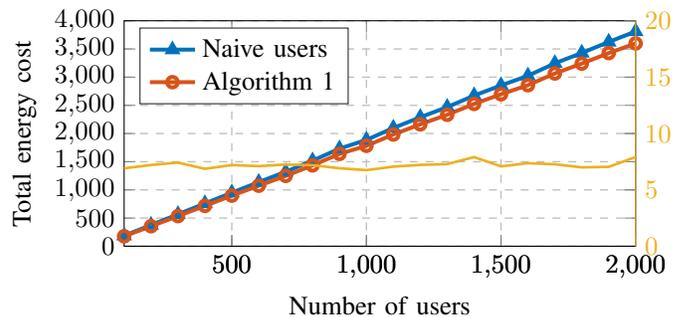
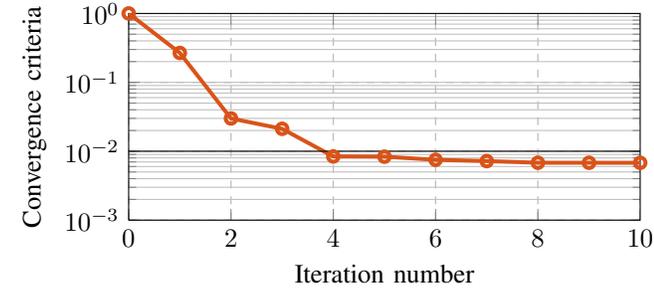
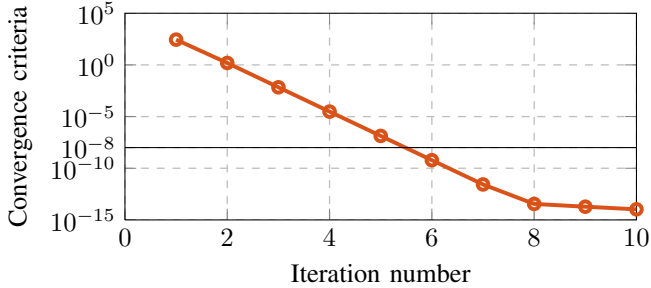
\begin{figure}
	\centering
	\begin{subfigure}[b]{1\linewidth}
    		\centering
	    	\input{conv_alg1_1000users.tikz}
		\caption{Convergence rate of Algorithm 1 with convergence criteria 
		$\Vert\bm{l}^i-\bm{l}^{i-1}\Vert_2/\Vert \bm{l}^i\Vert\leq 10^{-2}$.}
		\label{fig:conv-alg1-1000users}
	\end{subfigure}
	\begin{subfigure}[b]{1\linewidth}
	 	\centering
    		\input{convergence.tikz}
		\caption{Convergence rate of $\bm{\delta}^k(h)$ from step~4 in Algorithm~\ref{alg:worst-case-game} with convergence criteria $\sum_{h\in H}\Vert \bm{\delta}^{k+1}(h)-\bm{\delta}^k(h)\Vert\leq 10^{-8}$.}
		\label{fig:convergence}
	\end{subfigure}
	\caption{Convergence rates for 1000 users.}
\label{fig:fig}
\end{figure}

\vspace{-0.5em}
\section{Conclusion}
\label{sec:conclusion}
We presented a demand-side management model that takes into account deviations in the estimated consumption profiles, and analyzed the worst-case effects of these terms.
We considered a scenario where the error terms are calculated by the energy users (yielding a min-max game) which includes non-concave maximizing problems for all users.
This method provides estimates of the worst-case energy deviations and aggregate loads of the network, and further considers robust pricing as an effect of both global energy demand, as well as uncertain real-time demand costs of the supply authority.
The problems can be solved globally due to the quadratic structure they present in the objective and constraints.
With a novel analysis we are able to prove existence of NE, as well as to propose a novel convergent distributed algorithm to find such solutions.
The proposed algorithm has very little computational complexity, and provides significant monetary savings (around the 7\% in our simulations).
We further illustrated the benefit of our robust model giving flexibility to the users within certain margins while providing a robust production price to the energy supplier.

\appendices
\section{}
\label{ap:proof-th1}
\begin{proof}[Proof of Theorem \ref{th:theorem-gameA}]
We first need to establish that function $f_n(\bm{l}_n,\bm{l}_{-n},\bm{\delta})$ is convex provided that $\bm{l}_n$ is a convex function.
Indeed, this is satisfied because of the scalar composition rule~\cite[Sec.3.2.4]{Boyd2004}, where $f_n(\bm{l}_n,\bm{l}_{-n},\bm{\delta})$ is monotone if $L>0$, as required by~\eqref{eq:total-load}.
Then, because $\Omega_{\bm{l}_n}$ is convex by assumption, the individual user problem is also convex.

The rest of the results we establish are similar to the ones considered in \cite[Th.1 \& Th.2]{Atzeni2012}. 
These are based on the equivalence of game $\gameA$ to a variational inequality, as explained in \cite{Scutari2012}.
To show \emph{a)} is satisfied, it is sufficient to verify that the regions $\Omega_{\bm{l}_n}$ are compact and convex, which is satisfied by definition of $\gameA$, and that the objective functions of all players are differentiable and convex in variables $\bm{l}_n$, $\forall n\in \mathcal{D}$.
Indeed, the convexity is guaranteed since the objective functions are quadratic with positive terms in the squared elements so that the Hessian results into a diagonal matrix of the form $\diag(\{K_h\}_h)$ that is positive definite.
Hence, the claim is satisfied applying \cite[Th.4.1]{Scutari2012}.

The justification of \emph{b)} is readily available by observing that the objective is convex in $l_n(h)$ and that other players actions only affect the individual objective through the aggregate $L(h)$. 
Therefore, the maximum value is unique for each player given any strategy profile $\bm{l}^\ast_{-n}$ since the maximum does not depend on individual decisions. 
This shows that even if there are multiple NEs that reach the maximum of $f_n(\bm{l}_n,\bm{l}_{-n})$, the objective value of the user remains constant.

The game is monotone if the Jacobian $JF_{\gameA}(\bm{l},\bm{\delta})$ is positive semidefinite , where $JF_{\gameA}(\bm{l},\bm{\delta}) = \big(\nabla_{\bm{l}_m} F_n(\bm{l},\bm{\delta})\big)_{n,m\in\mathcal{D}}$,  $F_n = \big(\nabla_{\bm{l}_n} f_n(\bm{l}_n,\bm{l}_{-n},\bm{\delta})\big)_{n \in \mathcal{D}}$ and where $f_n(\bm{l}_n,\bm{l}_{-n},\bm{\delta})$ is described in \eqref{eq:indiv_cost}. The Hessian $\nabla^2_{\bm{l}_n,\bm{l}_n}f_n(\bm{l}_n,\bm{l}_{-n},\bm{\delta})=\diag(\{2K_h\}_{h\in H})$ is positive definite and the other terms $\nabla^2_{\bm{l}_m,\bm{l}_n}f_n(\bm{l}_n,\bm{l}_{-n},\bm{\delta})=\diag(\{K_h\}_{h\in H})$ for $m\neq n$ are also positive definite and, therefore, so is $JF_{\gameA}(\bm{l},\bm{\delta})$. 
The claim is satisfied applying \cite[Eq.(4.8)]{Scutari2012}.
\end{proof}

\section{}
\label{proof-theorem-monotoneBt}
\begin{proof}[Proof of Lemma \ref{lemma:lower-bound}]
We use the {\em extreme value theorem} from {\em calculus} (also known as the Bolzano-Weierstrass theorem), which states that a real-valued function on a nonempty compact space is bounded above, and attains its supremum. Then, Theorem~\ref{th:globalmax} states that if a maximum exists, there also exists a $\bar{\lambda}$ that satisfies equation \eqref{eq:posdef_cond} and, therefore, we necessarily have that there exists some $\lambda_h\geq K_h+\betamm$ that solves the problem.
\end{proof}

\begin{proof}[Proof of Theorem \ref{th:monotone-gameBt}]
To proof part \emph{a)} of the theorem (the existence of NE in $\gameBt$) we use a result from \cite[Th.1]{Nishimura1981}. 
We need to show that assumptions A1-A3 and A5 are satisfied from the mentioned reference (A4 is not required). 
Indeed, $\gameBt$ has a number of finite players, the regions for all players are compact and convex, and all objective functions are continuous on these regions (which satisfy A1-A3). 
Condition A5 is satisfied if we show that $p_n(\bm{l},\bm{\delta}_n,\bm{\delta}_{-n},\bm{\lambda})$ is concave in $\bm{\delta}_n$, which is true if its hessian is negative semidefinite. 
Indeed, $\nabla^2_{\bm{\delta}_n,\bm{\delta}_{n}} p_n = -2\text{diag}(\{\lambda_h - K_h -\betamm\})$ is diagonal, and is negative semidefinite if $\lambda_h \geq K_h + \betamm$. This also justifies that the {\em strong max-min property} of $\gameBt$ holds and the min-max can be transformed into a max-min without altering the result. This concludes the proof of part a) of the theorem.

To prove part \emph{b)} we need Lemma \ref{lemma:game-equivalence}, which guarantees the equivalence of $\gameB$ and $\gameBt$ when~\eqref{eq:wholeworst-tilde} is satisfied. Since by Lemma \ref{lemma:lower-bound} we have $\lambda_h\geq\betamm+K_h$ satisfied for all players, we conclude that a variational solution (a solution in which all $\lambda_h$ are shared among all users) exists. Note that the game $\gameB$ cannot satisfy condition A5 from previous paragraph directly and, therefore, the reformulation as $\gameBt$ is convenient to conclude the existence of an NE.

In order to prove part \emph{c)}, it suffices to show that the Jacobian $JF_{\gameBt}(\bm{\delta}(h))$ of the variational inequality associated with the game is negative definite~\cite[(4.8)]{Scutari2012}.
This is defined as $JF_{\gameBt}(\bm{l},\bm{\delta}) = \big(\nabla_m F_n(\bm{l},\bm{\delta})\big)_{n,m\in\mathcal{D}}$, where  $F_n = \big(\nabla_{\bm{\delta}_n} p_n(\bm{l},\bm{\delta}_n,\bm{\delta}_{-n},\bm{\lambda})\big)_{n \in \mathcal{D}}$ and $p_n(\bm{l},\bm{\delta}_n,\bm{\delta}_{-n},\bm{\lambda})$ is described in~\eqref{eq:minmax-game-pn}.
The Hessian for each $h\in H$ is diagonal on all players, and the Jacobian on each $h\in H$ becomes
\begin{equation}
	\nabla^2_{\bm{\delta}(h),\bm{\delta}(h)}p_n(\bm{l},\bm{\delta},\bm{\lambda})= (K_h+2\betamm-2\lambda_h) \mathbf{I} + K_h \mathbf{1}\mathbf{1}^T
\end{equation}
where $\mathbf{I}$ is the identity matrix of size $|\mathcal{D}|\times|\mathcal{D}|$, and $\mathbf{1}$ is column vector of ones of length $|\mathcal{D}|$. 
The Hessian matrix becomes diagonally dominant with negative diagonal entries for $\lambda_h > \frac{1}{2}K_h(|\mathcal{D}+1|)+\betamm$ and, thereby, negative definite~\cite[Cor.7.2.3]{Horn2012}.
If such condition is satisfied, then the game is strongly monotone and it has a unique NE that can be reached by a suitable algorithm~\cite{Scutari2012}.
\end{proof}

\vspace{-1em}
\section{}
\label{th:minmax-fixpoint-proof}
For notation simplicity, we refer in this section $T_h(\bm{\delta}(h),\bm{a}_h)$ defined in~\eqref{eq:minmax-fixedpoint-map} simply as $T(\x)=\Pi_X\big(\sqrt{\alpha}\tfrac{A\x+\bm{a}}{\Vert A\x+\bm{a}\Vert}\big)$, where $\bm{a}=\bm{a}_h$ is of size $|\mathcal{D}|\times1$, $A=\bm{1}\bm{1}^{T}-\mathbf{I}$ of size $|\mathcal{D}|\times |\mathcal{D}|$, $x=\bm{\delta}(h)$ and where we drop all indexes on~$h$.

Before proving Theorem~\ref{th:minmax-fixpoint-mapping} we first introduce the following general result:
\begin{lemma}
\label{th:lemma-normed-vectors}
Given any two vectors $x_1,x_2\in \mathbb{R}^{|\mathcal{D}|}$, the following holds:
\begin{IEEEeqnarray*}{c}
	\left\Vert \frac{x_1}{\Vert x_1 \Vert} -\frac{x_2}{\Vert x_2 \Vert}\right\Vert
	\leq  \left\Vert \frac{x_1}{\gamma_s} 
	-\frac{x_2}{\gamma_s}\right\Vert  
	 \leq  \left\Vert \frac{x_1}{\gamma} -\frac{x_2}{\gamma}\right\Vert
\end{IEEEeqnarray*}
where $\gamma_s\triangleq\min\{\Vert  x_1 \Vert,\Vert  x_2 \Vert\}$ and $\gamma \triangleq \min_{x\in X}\Vert x\Vert$, assuming in the last inequality that $\gamma_s \neq 0\neq \gamma$.
\end{lemma}
\begin{proof}
We start with the first inequality and assume without loss of generality that $\Vert x_1\Vert \geq \Vert x_2\Vert$. 
Introduce the shorthands $a=\frac{x_1}{\Vert x_1 \Vert}-\frac{x_2}{\Vert x_2 \Vert}$,  
$b=\frac{x_1}{\Vert x_2 \Vert}-\frac{x_1}{\Vert x_1 \Vert}$ and, therefore, $a+b=\frac{x_1}{\Vert x_2 \Vert}-\frac{x_2}{\Vert x_2 \Vert}$.
By the definition of scalar product, $\Vert a+b \Vert^2 = \Vert a \Vert^2+\Vert b \Vert^2+2a^Tb$.
Hence, if $a^T b\geq 0$, we necessarily have $\Vert a+b \Vert \geq \Vert a\Vert$, since all the sums involve positive scalars.
Then, the inner product results:
\begin{IEEEeqnarray}{rCl}
	a^T b &=& \left(\frac{x_1}{\Vert x_1\Vert}-\frac{x_2}{\Vert x_2\Vert}\right)^T \left(\frac{x_1}{\Vert x_2\Vert}-\frac{x_1}{\Vert x_1\Vert} \right) \\
	& = & \frac{\Vert x_1 \Vert}{\Vert x_2\Vert}-1-\frac{x_2^T x_1}{\Vert x_2\Vert^2}+\frac{x_2^T x_1}{\Vert x_1\Vert \Vert x_2\Vert}\\
	& = & {\left(\frac{\Vert x_1 \Vert}{\Vert x_2\Vert}-1\right)}{\left(1- \cos\theta\right)} \geq 0
\end{IEEEeqnarray}
where we have used $\cos\theta = \frac{x_2^Tx_1}{\Vert x_1\Vert\Vert x_2\Vert}$.

The proof of the second inequality is more immediate, since $\frac{\gamma_s}{\gamma} \geq 1$ we have $\Vert a+b\Vert\leq \left\Vert \frac{x_1}{\gamma} -\frac{x_2}{\gamma}\right\Vert$.
A graphical interpretation of this lemma is presented in Figure~\ref{fig:graphical-interpretation-lemma}.
\end{proof}

Now we use Lemma~\ref{th:lemma-normed-vectors} to prove Theorem \ref{th:minmax-fixpoint-mapping}.
\begin{proof}[Proof of Theorem \ref{th:minmax-fixpoint-mapping}]
It is clear by the projection operator $\Pi_{X_h}(\cdot)$ that $T(\x)$ is a self-map onto itself, so \emph{a)} is immediate.

To prove part \emph{b)} we need to show that $T(\x)$ is a contraction:
\begin{equation} 
	\left\Vert T(\x_{1})-T(\x_{2})\right\Vert \leq q\left\Vert \x_{1}-\x_{2}\right\Vert 
	\label{eq:contraction-mapping-def}
\end{equation}
for all $\x_{1},\x_{2}\in X$, and some $q\in(0,1)$. We have
\begin{IEEEeqnarray*}{c}
	\left\Vert \sqrt{\alpha}\tfrac{\a+A\x_{1}}{\left\Vert \a+A\x_{1}\right\Vert }-\sqrt{\alpha}\tfrac{\a+A\x_{2}}
	{\left\Vert \a+A\x_{2}\right(\Vert }\right\Vert 
	\overset{(i)}{\leq}  \sqrt{\alpha}\left\Vert \tfrac{\a+A\x_{1}}
	{\gamma}-\tfrac{\a+A\x_{2}}{\gamma}\right\Vert  \\
	 \leq \tfrac{\sqrt{\alpha}}{\gamma}\left\Vert A\right\Vert \left\Vert \x_{1}-\x_{2}\right\Vert
	 = \tfrac{\sqrt{\alpha}}{\gamma}(|\mathcal{D}|-1) \left\Vert \x_{1}-\x_{2}\right\Vert
\end{IEEEeqnarray*}
where in step $(i)$ we have applied the result from Lemma \ref{th:lemma-normed-vectors} and defined $\gamma=\min_{\x\in X}\left\Vert \a+A\x\right\Vert$.

At this point, we can infer that~\eqref{eq:contraction-mapping-def} is satisfied with $q=\tfrac{\sqrt{\alpha}}{\gamma}(|\mathcal{D}|-1)$ and hence, $T(\x)$ is a contraction mapping in region $X$ if $\frac{\sqrt{\alpha}}{\gamma}(|\mathcal{D}|-1)<1$ is satisfied. 
In order to prove this last condition, we analyze the quadratic form $\Vert A\x+\a\Vert >\sqrt{\alpha}(|\mathcal{D}|-1)$, which is depicted in Figure~\ref{fig:graphical-interpretation-theorem}. 
This inequality forms a paraboloid with minimum point in $x=-A^{-1}\a$ and curvature specified by it's Hessian $(A^T A)^{1/2}$. 
This Hessian matrix has maximum eigenvalue $\lambda_{\max} = (|\mathcal{D}|-1)$ with associated eigenvector $v_{\max} = \frac{\bm{1}}{\Vert\bm{1}\Vert}$, and the rest of eigenvalues are equal with value $\lambda_{\min} = 1$. 
The level curve corresponding to the value of $\sqrt{\alpha}(|\mathcal{D}|-1)$ is therefore an ellipsoid that grows faster in the direction of $v_{\max}$. 
Then, one sufficient condition to satisfy that the mapping $T$ is always a contraction, is to construct the region $X$ outside of the ellipsoid.
We define $X$ as the upper halfspace limited by the hyperplane with normal vector $v_{\max}$ and most distant point $x=A^{-1}\a+\sqrt{\alpha}\frac{\bm{1}}{\Vert\bm{1}\Vert}$. 
Such region is defined in equation~\eqref{eq:set-X} and, since $X$ is convex, the euclidean projection of $T$ is well defined.
This proves \emph{b)}.

The convergence property of the iterative fixed-point equation directly follows from the Banach fixed point theorem \cite[Th.11.1.6.]{Davidson2009}, which states that the mapping admits a unique fixed point $\x^\ast = T(\x^\ast)$ in $\x\in X$, and that such point can be found as the limit point of the sequence $\x^{k+1} = T(\x^{k})$, for any starting point $\x^0\in X$.  This proves \emph{c)}.

Finally, to prove \emph{d)}, it is sufficient to show that the following problem \textbf{P} is feasible:
\begin{IEEEeqnarray*}{r'l}
	\textbf{P}: &  
\begin{IEEEeqnarraybox}[][c]{r'l}
	\text{find} & \x \\
	\text{s.t.} & \x=\sqrt{\alpha}\tfrac{\a+A \x}{\Vert \a+A \x\Vert}  \\
			    & \bm{1}^T\x+\tfrac{1}{|\mathcal{D}|-1}\a^T\bm{1}-\sqrt{\alpha |\mathcal{D}|}\geq 0,
\end{IEEEeqnarraybox} 
\end{IEEEeqnarray*}
which would prove the existence of a point that solves~\eqref{eq:solgame_norm2} and is inside region $X$.
An equivalent formulation of \textbf{P} is
\begin{IEEEeqnarray*}{r'l}
	\text{find} & \x, z \\
	\text{s.t.} & \sqrt{\alpha}\bm{1}^T(\a+ A\x) +\big( \tfrac{1}{|\mathcal{D}|-1}\a^T \bm{1}-\sqrt{\alpha |\mathcal{D}|}    \big)z \geq 0\\
			    & \Vert A\x+\a\Vert = z,\quad z>0. 
\end{IEEEeqnarray*}
If $\tfrac{1}{|\mathcal{D}|-1}\a^T \bm{1}-\sqrt{\alpha |\mathcal{D}|} = 0$ the previous problem is unconditionally feasible.
Otherwise, it can be transformed as follows
\begin{IEEEeqnarray}{r'l}
	\text{find} & \bm{y}, z \IEEEnonumber \\
	\text{s.t.} & z \geq \tfrac{-\sqrt{\alpha}}{\tfrac{1}{|\mathcal{D}|-1}\a^T\bm{1}-\sqrt{\alpha |\mathcal{D}|}}\bm{1}^T\bm{y} 
				\label{eq:proof-plane} \\
			    & \Vert \bm{y}\Vert = z,\quad z>0. 
			    \label{eq:proof-cone}
\end{IEEEeqnarray}
Equation~\eqref{eq:proof-plane} represents a halfspace defined by a non vertical and non horizontal hyperplane that passes through the origin.
Equations in~\eqref{eq:proof-cone} represent a positive cone.
Since the intersection of these spaces is necessarily nonempty the problem is feasible, and so is \textbf{P}. This proves \emph{d)} and concludes the proof of the theorem.
\end{proof}
\begin{figure}
    \centering
    \includegraphics[{width=6.8cm},height=3.5cm]{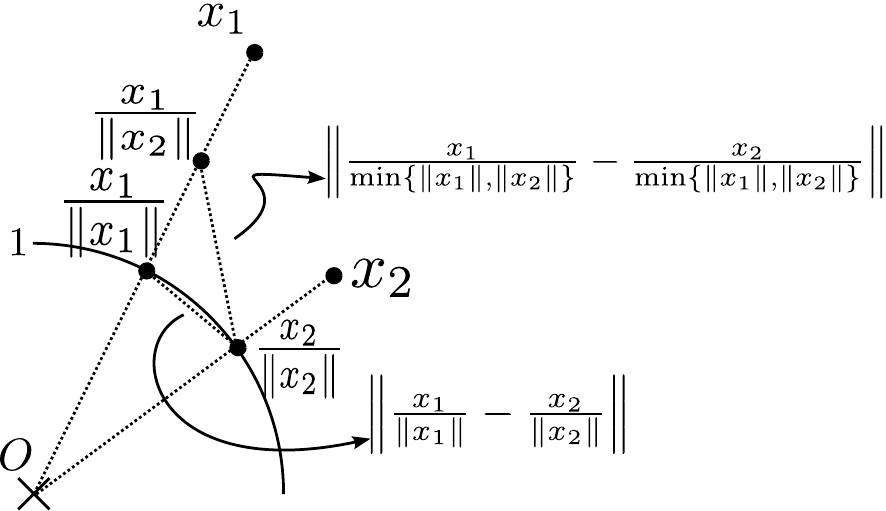}
	\caption{Graphical interpretation of Lemma~\ref{th:lemma-normed-vectors}}
	\label{fig:graphical-interpretation-lemma}
\end{figure}
\begin{figure}
    \centering
    \includegraphics[width=6.8cm,height=4.3cm]{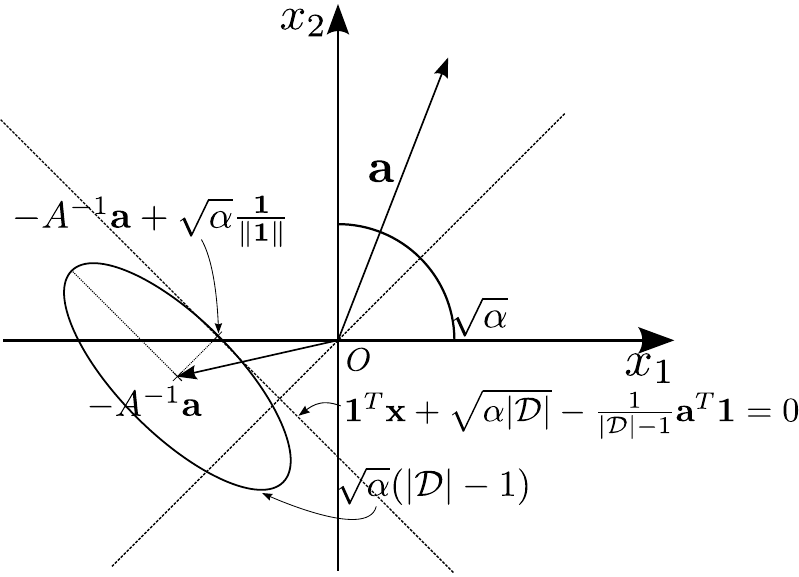}
	\caption{Graphical interpretation of Theorem~\ref{th:lemma-normed-vectors}}
	\label{fig:graphical-interpretation-theorem}
	\vspace{-1em}
\end{figure}

\pagebreak
\section*{Acknowledgment}
The authors would like to thank the anonymous reviewers for their comments, which helped to improve the manuscript.

\bibliographystyle{IEEEbib}
\bibliography{smartgrid}

%
\vspace{-3em}
\begin{IEEEbiography}[{\includegraphics[width=1in,height=1.25in,clip,keepaspectratio]{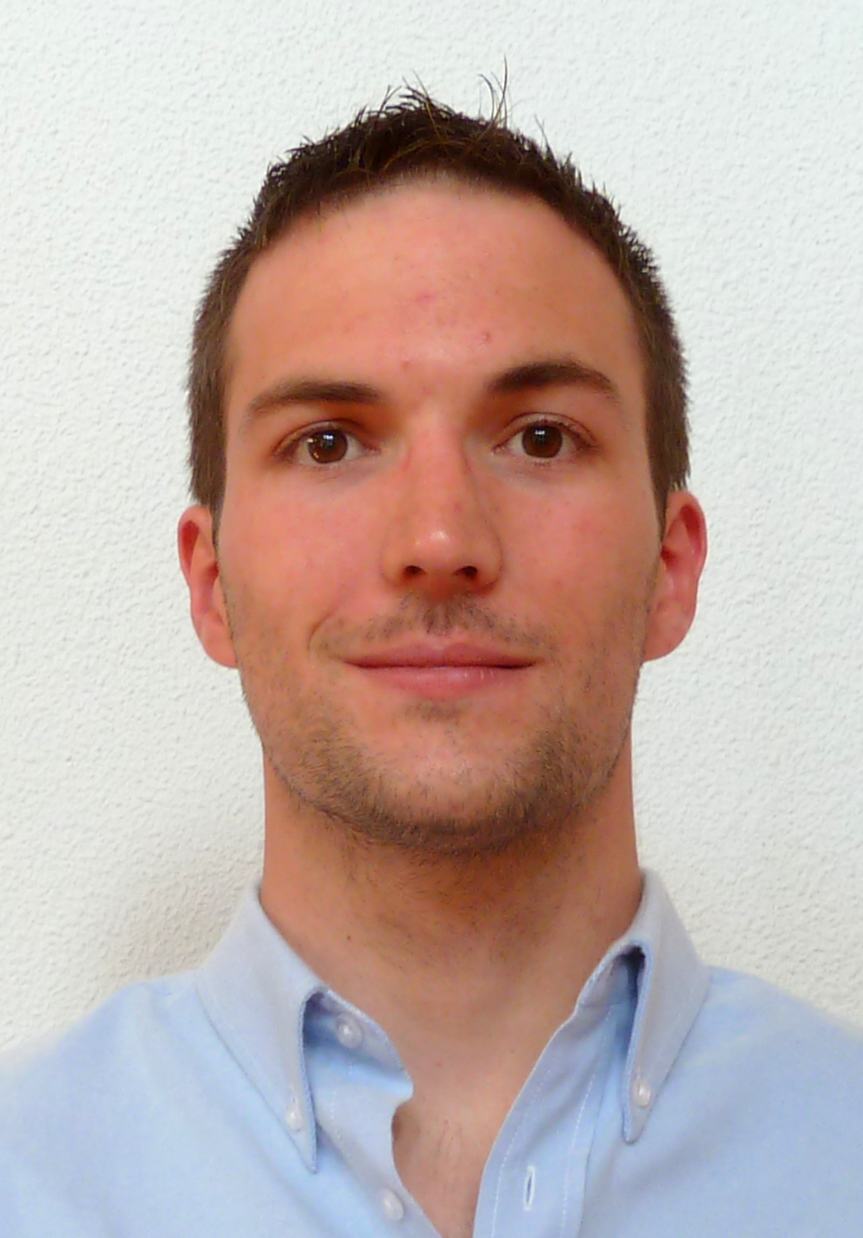}}]{Javier
Zazo}
(S'12)
received his Telecommunications Engineer degree from Universidad
Polit\'ecnica de Madrid (UPM) and Technische Universit\"at Darmstadt
(TUD) in 2010, and M.Sc. by Research from the National University of
Ireland, Maynooth (NUIM) in 2012.
He is currently pursuing a Ph.D. degree in Communications Systems at UPM,
funded by an FPU doctoral grant from the Spanish Ministry of Science and
Innovation.
His current research interests include distributed optimization, game
theory and variational inequalities, applied in areas such as wireless
networks and smart grid.
\end{IEEEbiography}

\vspace{-3em}
\begin{IEEEbiography}[{\includegraphics[width=1in,height=1.25in,clip,keepaspectratio]{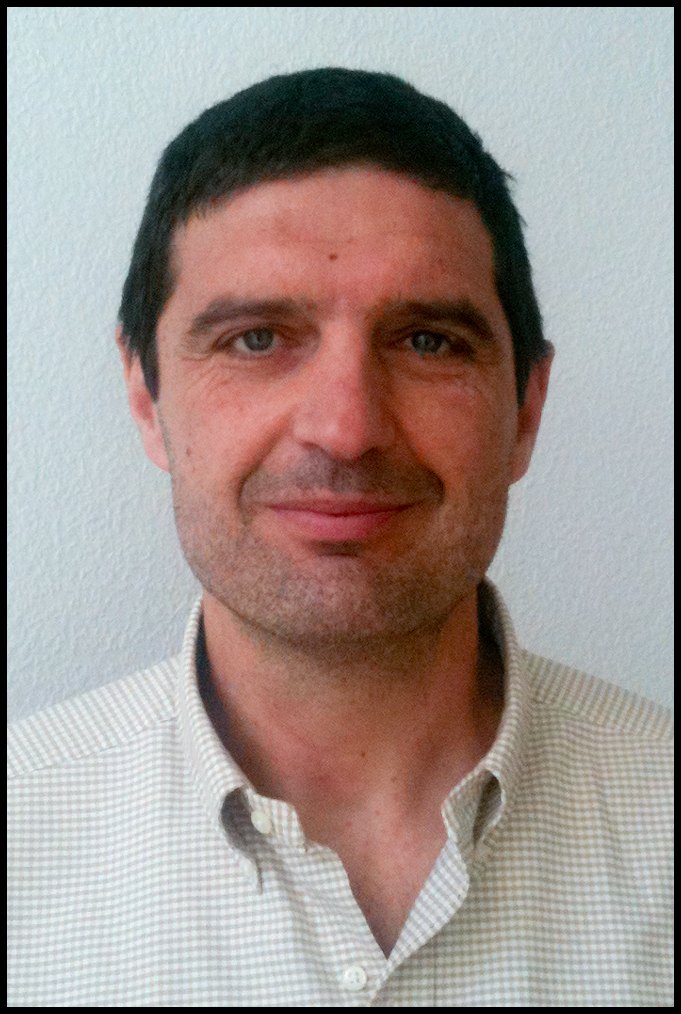}}]{Santiago
Zazo}
(M'12)
is Telecom Engineer and Dr. Engineer by Universidad Polit\'ecnica de
Madrid (UPM),
since 1990 and 1995 respectively.
From 1991 to 1994 he was with the University of Valladolid,
and with the University Alfonso X El Sabio, from 1995 to
1997.
In 1998, he joined UPM as Associate Professor in Signal Theory and
Communications.
His main research activities are in the field of Signal Processing,
with applications to audio, radar,
communications, MIMO and Wireless Sensor Networks.
His main activity at this moment is related to distributed algorithms,
optimization and game theory.
Since 1990 he has (co)authored more than 30 journal papers and about 150
conference papers.
He has also led many private and national Spanish projects and has
participated in five European projects, leading three of them.
\end{IEEEbiography}

\vspace{-3em}
\begin{IEEEbiography}[{\includegraphics[width=1in,height=1.25in,clip,keepaspectratio]{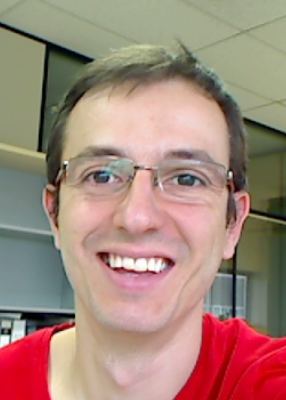}}]{Sergio
Valcarcel Macua}
(S'10)
received his B.S. in Telecommunications Engineering and M.S. in
Electronic Systems from Universidad Polit\'ecnica de Madrid (UPM),
where he is currently a Ph.D. candidate.
Before starting his Ph.D.,
he has been Research Associate at Robotics Institute, Carnegie Mellon
University,
and Research Visitor at University of California, Los Angeles,
and at INRIA Nord, Lille.
His research interests include reinforcement learning,
optimization,
machine learning and statistical signal processing
in distributed multiagent systems.
\end{IEEEbiography}




\end{document}

%% file: realtime_cost.tikz
%
\definecolor{mycolor1}{rgb}{0.00000,0.44700,0.74100}%
\definecolor{mycolor2}{rgb}{0.85000,0.32500,0.09800}%
\definecolor{mycolor3}{rgb}{0.92900,0.69400,0.12500}%
\begin{tikzpicture}

\begin{axis}[%
width=\figurewidthD,
height=\figureheightD,
at={(0.758in,0.481in)},
scale only axis,
xmin=100,
xmax=2000,
xlabel={Number of users},
xmajorgrids,
ymin=0,
ymax=4000,
ytick={   0,  500, 1000, 1500, 2000, 2500, 3000, 3500, 4000},
ylabel={Average Cost},
ymajorgrids,
axis background/.style={fill=white},
title style={font=\bfseries},
legend style={at={(0.03,0.97)},anchor=north west,legend cell align=left,align=left,draw=white!15!black}
]
\addplot [color=mycolor1,solid,line width=1.5pt,mark=triangle,]
  table[row sep=crcr]{%
100	9.40439247242477\\
200	37.9971391655293\\
300	84.4609264834764\\
400	153.021559912895\\
500	239.399939098256\\
600	339.113422802725\\
700	463.957605510809\\
800	611.979929078311\\
900	799.681968071735\\
1000	941.448447874173\\
1100	1168.52985638913\\
1200	1383.3459586267\\
1300	1618.24098953289\\
1400	1895.5600579374\\
1500	2157.35409983393\\
1600	2421.17788989356\\
1700	2789.00817083956\\
1800	3107.68286237953\\
1900	3499.82523231752\\
2000	3841.63627690608\\
};
\addlegendentry{Non-Robust};

\addplot [color=mycolor2,solid,line width=1.5pt,solid,mark=o]
  table[row sep=crcr]{%
100	7.22000401758522\\
200	28.9053893660369\\
300	64.1064714768231\\
400	116.494023948155\\
500	182.310385592917\\
600	257.90471518864\\
700	352.272090494387\\
800	465.352795909705\\
900	611.178729491165\\
1000	715.241028400839\\
1100	889.8427248099\\
1200	1053.39730717588\\
1300	1230.37311395661\\
1400	1443.13682729418\\
1500	1641.39968765883\\
1600	1839.09377990829\\
1700	2123.2435062938\\
1800	2363.8594661281\\
1900	2667.17319624438\\
2000	2921.75598535037\\
};
\addlegendentry{Algorithm 1};

\end{axis}

\begin{axis}[%
width=\figurewidthD,
height=\figureheightD,
at={(0.758in,0.481in)},
scale only axis,
xmin=100,
xmax=2000,
every outer y axis line/.append style={mycolor3},
every y tick label/.append style={font=\color{mycolor3}},
ymin=0,
ymax=100,
ytick={  0,   20,  40,  60,  80,  100},
axis x line*=bottom,
axis y line*=right
]
\addplot [color=mycolor3,solid,line width=1.5pt,forget plot]
  table[row sep=crcr]{%
100	30.2546709048804\\
200	31.4534763201459\\
300	31.7510144260742\\
400	31.355716565338\\
500	31.3144823426651\\
600	31.4878723929825\\
700	31.7043325401339\\
800	31.5088110477491\\
900	30.8425718181502\\
1000	31.6267398668525\\
1100	31.3186952940215\\
1200	31.3223367102961\\
1300	31.5244108617579\\
1400	31.3499885864253\\
1500	31.4338071375545\\
1600	31.6505942407294\\
1700	31.3560202855804\\
1800	31.4664812739388\\
1900	31.2185214385622\\
2000	31.4838164503805\\
};
\end{axis}
\end{tikzpicture}%

%% file: mm2000u24h.tikz
%
%
\definecolor{mycolor1}{rgb}{0.00000,0.44700,0.74100}%
\definecolor{mycolor2}{rgb}{0.85000,0.32500,0.09800}%
\definecolor{mycolor3}{rgb}{0.92900,0.69400,0.12500}%
\begin{tikzpicture}

\begin{axis}[%
width=\figurewidthB,
height=\figureheightB,
at={(0.771875in,0.483542in)},
scale only axis,
xlabel={Number of users},
ylabel={Total energy cost},
xmin=100,
xmax=2000,
xmajorgrids,
ymin=0,
ymax=4000,
ytick={   0,  500, 1000, 1500, 2000, 2500, 3000, 3500, 4000},
ymajorgrids,
legend style={at={(0.03,0.97)},anchor=north west,legend cell align=left,align=left,draw=white!15!black}
]
\addplot [color=mycolor1,solid,line width=1.5pt,mark=triangle]
  table[row sep=crcr]{%
100	187.182195309017\\
200	376.152024042778\\
300	566.643372861118\\
400	758.361127482576\\
500	949.761815291969\\
600	1137.24837427836\\
700	1321.65154448868\\
800	1518.27154618576\\
900	1731.87070411965\\
1000	1886.09459539086\\
1100	2097.87977163315\\
1200	2288.20174943484\\
1300	2468.47092189001\\
1400	2672.93825498226\\
1500	2853.35546173391\\
1600	3023.90894134046\\
1700	3245.1887236921\\
1800	3428.89716670192\\
1900	3624.27228676651\\
2000	3810.12689129809\\
};
\addlegendentry{Naive users};

\addplot [color=mycolor2,solid,line width=1.5pt,mark=o,mark options={solid},]
  table[row sep=crcr]{%
100	177.148505389668\\
200	355.385522039274\\
300	535.296339042203\\
400	716.4010911236\\
500	897.856139038223\\
600	1074.18174323372\\
700	1248.24734929704\\
800	1434.93720141532\\
900	1638.33494683135\\
1000	1782.04716521945\\
1100	1982.78906754723\\
1200	2162.85248646961\\
1300	2332.80828476687\\
1400	2525.21100909988\\
1500	2696.96537445331\\
1600	2856.9628235564\\
1700	3065.31829567915\\
1800	3240.78441526506\\
1900	3426.59315890292\\
2000	3598.45665783332\\
};
\addlegendentry{Algorithm 1};

\end{axis}

\begin{axis}[%
width=\figurewidthB,
height=\figureheightB,
at={(0.771875in,0.483542in)},
scale only axis,
xmin=100,
xmax=2000,
every outer y axis line/.append style={mycolor3},
every y tick label/.append style={font=\color{mycolor3}},
ymin=0,
ymax=20,
ytick={ 0,  5, 10, 15, 20},
axis x line*=bottom,
axis y line*=right
]
\addplot [color=mycolor3,solid,forget plot,line width=1pt]
  table[row sep=crcr]{%
100	6.91176583487433\\
200	7.20741085665551\\
300	7.42343874051777\\
400	6.86843099101754\\
500	7.20391069261781\\
600	7.10772039365399\\
700	7.2482537637812\\
800	7.19467203270474\\
900	6.91817798732917\\
1000	6.75326943968563\\
1100	7.06240812892998\\
1200	7.21734876060714\\
1300	7.28968625332312\\
1400	7.90759082160772\\
1500	7.09017310510534\\
1600	7.36748732372717\\
1700	7.27494845481047\\
1800	6.99992092301473\\
1900	7.04011138483411\\
2000	7.91005335384835\\
};
\end{axis}
\end{tikzpicture}%

%% file: conv_alg1_1000users.tikz
%
%

\definecolor{mycolor1}{rgb}{0.00000,0.44700,0.74100}%
\definecolor{mycolor2}{rgb}{0.85000,0.32500,0.09800}%
\definecolor{mycolor3}{rgb}{0.92900,0.69400,0.12500}%

\begin{tikzpicture}

\begin{axis}[%
width=\figurewidthE,
height=\figureheightE,
at={(0.758333in,0.48125in)},
scale only axis,
xmin=0,
xmax=10,
xlabel={Iteration number},
xmajorgrids,
ymode=log,
ymin=0.001,
ymax=1,
yminorticks=true,
ylabel={Convergence criteria},
ymajorgrids,
yminorgrids
]
\addplot [color=mycolor2,solid,line width=1.5pt,mark=o,mark options={solid},forget plot]
  table[row sep=crcr]{%
0	1\\
1	0.266953229671377\\
2	0.0299283549303286\\
3	0.0211694395905286\\
4	0.00840048532018193\\
5	0.00834079340904339\\
6	0.00751115947530836\\
7	0.00719608309705565\\
8	0.00681425098891992\\
9	0.00680643855927499\\
10	0.00679868601144397\\
};
\addplot [color=black,solid,forget plot]
  table[row sep=crcr]{%
0	0.01\\
1	0.01\\
2	0.01\\
3	0.01\\
4	0.01\\
5	0.01\\
6	0.01\\
7	0.01\\
8	0.01\\
9	0.01\\
10	0.01\\
};
\end{axis}
\end{tikzpicture}%

%% file: convergence.tikz
%
%

\definecolor{mycolor1}{rgb}{0.00000,0.44700,0.74100}%
\definecolor{mycolor2}{rgb}{0.85000,0.32500,0.09800}%
\definecolor{mycolor3}{rgb}{0.92900,0.69400,0.12500}%

\begin{tikzpicture}

\begin{axis}[%
width=\figurewidthC,
height=\figureheightC,
at={(0.924896in,0.438194in)},
scale only axis,
xmin=0,
xmax=10,
xlabel={Iteration number},
xmajorgrids,
ymode=log,
ymin=1e-15,
ymax=100000,
ytick={ 1e-15,  1e-10,  1e-08,  1e-05,      1, 100000},
yminorticks=true,
ylabel={Convergence criteria},
ymajorgrids,
yminorgrids
]
\addplot [color=mycolor2,solid,line width=1.5pt,mark=o,mark options={solid},forget plot]
  table[row sep=crcr]{%
1	278.048419199578\\
2	1.54541674581852\\
3	0.00682002615355241\\
4	3.02977218215456e-05\\
5	1.35002122911287e-07\\
6	6.03359917583646e-10\\
7	2.70632127818971e-12\\
8	3.52634588196565e-14\\
9	1.90680804479371e-14\\
10	1.06303854607859e-14\\
};
\addplot [color=black,solid,forget plot]
  table[row sep=crcr]{%
0	1e-08\\
1	1e-08\\
2	1e-08\\
3	1e-08\\
4	1e-08\\
5	1e-08\\
6	1e-08\\
7	1e-08\\
8	1e-08\\
9	1e-08\\
10	1e-08\\
};
\end{axis}
\end{tikzpicture}%